\documentclass[12pt]{amsart}
\usepackage{amssymb,amsmath,amsthm,mathrsfs,multirow,xcolor,framed,url}
\usepackage[pdftex,
         pdfauthor={Dandan Chen},
         pdftitle={},
         pdfsubject={MATHEMATICS},
         pdfkeywords={ Crank for partitions, Theta-functions.},
         pdfproducer={Latex with hyperref},
         pdfcreator={pdflatex}]{hyperref}
\usepackage [latin1]{inputenc}
\newtheorem{theorem}{Theorem}[section]
\newtheorem{lemma}[theorem]{Lemma}

\theoremstyle{definition}
\newtheorem{definition}[theorem]{Definition}

\theoremstyle{remark}

\numberwithin{equation}{section}

\newcommand\nutwid{\overset {\text{\lower 3pt\hbox{$\sim$}}}\nu}

\newcommand\stroke[2]{{#1}\,\left\arrowvert\,#2\right.}

\newcommand\MAT[4]{\begin{pmatrix} {#1} & {#2} \\ {#3}  & {#4} \end{pmatrix}}

\allowdisplaybreaks

\newcommand\omycite[1]{}

\newcommand{\beqs}{\begin{equation*}}
\newcommand{\eeqs}{\end{equation*}}
\newcommand{\beq}{\begin{equation}}
\newcommand{\eeq}{\end{equation}}
\renewcommand{\MR}[1]{\href{http://www.ams.org/mathscinet-getitem?mr={#1}}{MR{#1}}}

\begin{document}
\title[Congruences modulo powers of $7$ ]{Congruences Modulo Powers of 7 for the Reciprocal Crank Parity Function }


\author{Dandan Chen}
\address{Department of Mathematics, Shanghai University, People's Republic of China}
\address{Newtouch Center for Mathematics of Shanghai University, Shanghai, People's Republic of China}
\email{mathcdd@shu.edu.cn}




\date{}



\begin{abstract}
Amdeberhan and Merca recently studied arithmetic properties of the sequence $a(n)$, the reciprocal of the crank parity function, which counts the number of integer partitions of weight $n$ whose even parts are monochromatic and whose odd parts may appear in one of three colors (OEIS A298311). A key result of their work was the congruence $a(7n + 2) \equiv 0 \pmod{7}$ for all $n \geq 0$.
We prove new congruences for the reciprocal crank parity function modulo powers of $7$.
\end{abstract}

\maketitle


\section{Introduction}

A \textbf{partition} of a positive integer $n$ is a non-increasing sequence of positive integers whose sum is $n$ \cite{And76}. Let $p(n)$ denote the number of partitions of $n$, with the convention that $p(0) = 1$ and $p(n) = 0$ when $n$ is not a non-negative integer.
In 1919, Ramanujan \cite{Ram19} announced three elegant congruences satisfied by the partition function $p(n)$. These results reveal a remarkable arithmetic regularity.

\begin{theorem}[Ramanujan's Congruences]
    For every non-negative integer $n$, the partition function satisfies:
    \begin{align*}
        p(5n + 4) &\equiv 0 \pmod{5}, \\
        p(7n + 5) &\equiv 0 \pmod{7}, \\
        p(11n + 6) &\equiv 0 \pmod{11}.
    \end{align*}
\end{theorem}

To provide combinatorial explanations for the latter two congruences, Dyson \cite{Dys44} introduced the concept of the \textbf{rank} of a partition.

\begin{definition}[Rank]
    The \textbf{rank} of a partition is defined as its largest part minus the number of its parts.
\end{definition}

Later, in 1988, Andrews and Garvan \cite{AG88} defined the \textbf{crank} of a partition, which provides a unified combinatorial explanation for all three of Ramanujan's congruences.

\begin{definition}[Crank]
    Let $\lambda = (\lambda_1, \lambda_2, \dots, \lambda_k)$ be a partition. Define:
    \begin{itemize}
        \item $\ell(\lambda)$: the largest part of $\lambda$,
        \item $\omega(\lambda)$: the number of $1$'s in $\lambda$,
        \item $\mu(\lambda)$: the number of parts of $\lambda$ larger than $\omega(\lambda)$.
    \end{itemize}
    The \textbf{crank} $c(\lambda)$ is given by:
    \[
    c(\lambda) =
    \begin{cases}
        \ell(\lambda), & \text{if } \omega(\lambda) = 0, \\
        \mu(\lambda) - \omega(\lambda), & \text{if } \omega(\lambda) > 0.
    \end{cases}
    \]
\end{definition}

\begin{definition}
Let $n$ be a non-negative integer. We define:
\begin{enumerate}
    \item $M_e(n)$: the number of partitions of $n$ with even crank,
    \item $M_o(n)$: the number of partitions of $n$ with odd crank,
    \item $M(n) := M_e(n) - M_o(n)$.
\end{enumerate}
\end{definition}

From \cite{Ga}, we have the generating function identity:
\begin{equation*}
\sum_{n=0}^{\infty} M(n)q^n = 2q + \frac{(q;q)_\infty}{(-q;q)^2_\infty},
\end{equation*}
where $(a;q)_\infty$ denotes the standard $q$-Pochhammer symbol, defined by the infinite product:
\[
(a;q)_\infty = \prod_{n=0}^{\infty} (1 - a q^n).
\]
Here and later, $q$ is a complex number with $|q| < 1$. We also use notaion $J_k=(q^k;q^k)_\infty$ for integer $k>0$.

Recent work by Amdeberhan and Merca~\cite{AM-arxiv} has examined arithmetic properties of the sequence $a(n)$, which is defined as the reciprocal of the crank parity function arising from the generating function of $M_e(n) - M_o(n)$:
\begin{equation}\label{def-a}
    \sum_{n=0}^{\infty} a(n)q^n = \frac{(-q;q)^2_\infty}{(q;q)_\infty}.
\end{equation}
A key combinatorial interpretation of $a(n)$ is that it enumerates the number of integer partitions of weight $n$ wherein even parts are monochromatic (i.e., appear in only one color), while odd parts may appear in one of three colors~(see OEIS A298311). This interpretation, along with several others, is presented in the paper by Amdeberhan and Merca~\cite{AM-arxiv}.

Furthermore, they~\cite{AM-arxiv} utilized the Mathematica package \textsc{RaduRK}, developed by Smoot~\cite{S21}, to prove the following generating function identity for $a(7n + 2)$.

More recently, Hirschhorn and Sellers~\cite{HS-arxiv} considered a generalization of the partition function $a(n)$. For any integer $r \geq 1$, they defined $a_r(n)$ as the number of partitions of $n$ in which each odd part may be assigned one of $r$ colors. The generating function for $a_r(n)$ is given by
\begin{equation}
\sum_{n=0}^{\infty} a_r(n)q^n = \frac{(q^{2};q^{2})_\infty^{r-1}}{(q;q)^{r}_\infty}.
\label{eq:gen_func}
\end{equation}
Note that $a_1(n) = p(n)$ (the ordinary partition function), $a_2(n) = \overline{p}(n)$ (the number of overpartitions of $n$~\cite{CL04}), and $a_3(n) = a(n)$.

Hirschhorn and Sellers~\cite{HS-arxiv} also employed theta function identities and $q$-series manipulations to establish arithmetic congruences modulo 7 for $a_k(n)$.

In this paper, we establish the following theorem.
\begin{theorem}\label{thm-main}
For $\alpha \geq 0$, we have
\[
a(7^\alpha n + \lambda_\alpha) \equiv 0 \pmod{7^{\lfloor\frac{\alpha+1}{2}\rfloor}} \quad \text{if } 24\lambda_\alpha \equiv -1 \pmod{7^{\alpha}}.
\]
\end{theorem}
\section{Modular equation}
\label{sec-lemma}

Our proof of Theorem \ref{thm-main} relies on the modular identities in the Appendix.  Many of these, specifically those in Groups $I$-$IV$, can be automatically verified using Garvan's MAPLE package ETA (see \eqref{r:eta} and \cite{gtutorial})
\begin{align}
\label{r:eta}
https://qseries.org/fgarvan/qmaple/ETA/
\end{align}

For example, the package  yields the identity for $L_1$:
\begin{align}\label{L-1}
L_1=&-7 p_1+2\cdot7 t-11\cdot7^{2} p_1t+8\cdot7 p_0t+29\cdot7^{2} t^{2}-23\cdot 7^{3} p_1t^{2}\\
&+12\cdot 7^{3} p_0t^{2}+10\cdot 7^{4} t^{3}-2\cdot 7^{5} p_1t^{3}+24\cdot 7^{4} p_0 t^{3}
+7^{6}t^{4}+2\cdot 7^{6} p_0t^{4}.\nonumber
\end{align}

\subsection{\textbf{A modular equation}}

We define
\begin{align}
\label{eq:t7def}
t := t(\tau) := q \frac{J_7^4}{J_1^4},
\end{align}
where \( q = \exp(2\pi i\tau) \). Note that \( t(\tau) \) is a Hauptmodul for \( \Gamma_0(7) \) \cite{Ma09}.

The following result from \cite[Theorem 2.6]{Chen-Chen-Garvan-arxiv} will be used later.

\begin{theorem}\label{thm-ai}
Define
\begin{align*}
a_0(t) &= t, \\
a_1(t) &= 7^2 t^2 + 4 \cdot 7 t,  \\
a_2(t) &= 7^4 t^3 + 4 \cdot 7^3 t^2 + 46 \cdot 7 t,  \\
a_3(t) &= 7^6 t^4 + 4 \cdot 7^5 t^3 + 46 \cdot 7^3 t^2 + 272 \cdot 7 t,  \\
a_4(t) &= 7^8 t^5 + 4 \cdot 7^7 t^4 + 46 \cdot 7^5 t^3 + 272 \cdot 7^3 t^2 + 845 \cdot 7 t,  \\
a_5(t) &= 7^{10} t^6 + 4 \cdot 7^9 t^5 + 46 \cdot 7^7 t^4 + 272 \cdot 7^5 t^3 + 845 \cdot 7^3 t^2 + 176 \cdot 7^2 t, \\
a_6(t) &= 7^{12} t^7 + 4 \cdot 7^{11} t^6 + 46 \cdot 7^9 t^5 + 272 \cdot 7^7 t^4 + 845 \cdot 7^5 t^3 + 176 \cdot 7^4 t^2 + 82 \cdot 7^2 t,
\end{align*}
where \( t = t(\tau) \) is as in \eqref{eq:t7def}. Then
\[
t(\tau)^7 - \sum_{l=0}^{6} a_l\big(t(7\tau)\big) \, t(\tau)^l = 0.
\]
\end{theorem}

\subsection{The \( U_p \) Operator}

Let \( p \) be a prime and
\[
f = \sum_{m=m_0}^\infty a(m) q^m
\]
be a formal Laurent series. The \( U_p \) operator is defined by
\begin{equation}\label{Updeffls}
U_p(f) := \sum_{p m \ge m_0} a(p m) q^m.
\end{equation}
If \( f \) and \( h \) are modular functions (with \( q = \exp(2\pi i\tau) \)), then
\[
U_p(f) = \frac{1}{p} \sum_{j=0}^{p-1} \stroke{f}{\MAT{1}{j}{0}{p}}
 = \frac{1}{p} \sum_{j=0}^{p-1} f\left(\frac{\tau + j}{p}\right),
\]
and for \( H(\tau) = h(p\tau) \), we have
\begin{equation}\label{u71}
U_p(fH)(\tau) = h(\tau) U_p(f)(\tau).
\end{equation}

\begin{theorem}[{\cite[Lemma 7, p.138]{At-Le70}}]
Let \( p \) be prime. If \( f \) is a modular function on \( \Gamma_0(pN) \) and \( p \mid N \), then \( U_p(f) \) is a modular function on \( \Gamma_0(N) \).
\end{theorem}

\subsection{A Fundamental Lemma}
\label{subsec:fundlem5}

The following lemma is a direct consequence of Theorem \ref{thm-ai}.

\begin{lemma}[Fundamental Lemma]
\label{lem:fun7}
Let \( u = u(\tau) \) and \( j \in \mathbb{Z} \). Then
\[
U_7(u t^j) = \sum_{l=0}^{6} a_l(t) U_7(u t^{j+l-7}),
\]
where \( t = t(\tau) \) is defined in \eqref{eq:t7def} and the \( a_j(t) \) are given in Theorem \ref{thm-ai}.
\end{lemma}

From Theorem \ref{thm-ai}, we verify that there exist integers \( s(j,l) \) such that
\begin{equation}
a_j(t) = \sum_{l=1}^{7} s(j,l) 7^{\lfloor (7l+j-4)/4 \rfloor} t^l
\label{eq:aj}
\end{equation}
for \( 0 \le j \le 6 \).

Let \( g = \sum_{n} a_n t^n \), \( g \neq 0 \), where only finitely many \( a_n \) with \( n < 0 \) are nonzero. The order of \( g \) (with respect to \( t \)) is the smallest integer \( N \) such that \( a_N \neq 0 \), denoted \( N = \mathrm{ord}_t(g) \).

\begin{lemma}[{\cite[Lemma 3.6]{Chen-Chen-Garvan-arxiv}}]
\label{lem1}
Let \( u, v_1, v_2, v_3 : \mathbb{H} \rightarrow \mathbb{C} \) and \( l \in \mathbb{Z} \). Suppose that for \( l \leq k \leq l+6 \) and \( i = 1, 2, 3 \), there exist Laurent polynomials \( p_k^{(i)}(t) \in \mathbb{Z}[t,t^{-1}] \) such that
\begin{align}
U_7(u t^k) &= v_1 p_k^{(1)}(t) + v_2 p_k^{(2)}(t) + v_3 p_k^{(3)}(t), \label{eq:lem11} \\
\mathrm{ord}_t(p_k^{(i)}(t)) &\geq \left\lceil \frac{k + s_i}{7} \right\rceil, \label{eq:lem12}
\end{align}
for fixed integers \( s_i \). Then there exist families of Laurent polynomials \( p_k^{(i)}(t) \in \mathbb{Z}[t,t^{-1}] \), \( k \in \mathbb{Z} \), such that \eqref{eq:lem11} and \eqref{eq:lem12} hold for all \( k \in \mathbb{Z} \).
\end{lemma}

\begin{lemma}[{\cite[Lemma 3.7]{Chen-Chen-Garvan-arxiv}}]
\label{lem2}
Let \( u, v_1, v_2, v_3 : \mathbb{H} \rightarrow \mathbb{C} \) and \( l \in \mathbb{Z} \). Suppose that for \( l \leq k \leq l+6 \) and \( i = 1, 2, 3 \), there exist Laurent polynomials \( p_k^{(i)}(t) \in \mathbb{Z}[t,t^{-1}] \) such that
\begin{align}
U_7(u t^k) &= v_1 p_k^{(1)}(t) + v_2 p_k^{(2)}(t) + v_3 p_k^{(3)}(t), \label{eq:lem21}
\end{align}
where
\begin{align}
p_k^{(i)}(t) = \sum_{n} c_i(k,n) 7^{\left\lfloor \frac{7n - k + r_i}{4} \right\rfloor} t^n, \label{eq:lem22}
\end{align}
with integers \( r_i \) and \( c_i(k,n) \). Then there exist families of Laurent polynomials \( p_k^{(i)}(t) \in \mathbb{Z}[t,t^{-1}] \), \( k \in \mathbb{Z} \), of the form \eqref{eq:lem22} for which property \eqref{eq:lem21} holds for all \( k \in \mathbb{Z} \).
\end{lemma}

\section{Proof of Theorem \ref{thm-main}}
The proof relies on the forty-two fundamental relations listed in Appendix~\ref{funcr-7}.
These identities can be established using the algorithm described in \cite[Section~2C, pp.~8--9]{Ch-Ch-Ga20}.

From \eqref{def-a}, we have
\begin{equation*}
    \sum_{n=0}^{\infty} a(n)q^n = \frac{(-q;q)^2_\infty}{(q;q)_\infty}.
\end{equation*}

For a function $f : \mathbb{H} \rightarrow \mathbb{C}$, define operators $U_A(f)$ and $U_B(f) : \mathbb{H} \rightarrow \mathbb{C}$ by
\[
U_A(f) := U_7(A f), \quad U_B(f) := U_7(B f),
\]
where
\[
A := \frac{J_{2}^2 J_{49}^3}{q^2 J_1^3 J_{98}}, \quad B := 1.
\]

Define the initial function
\[
L_0 := 1,
\]
and set
\[
p_0 := \frac{q J_{14}^4 J_1^4}{J_7^4 J_2^4}, \quad
p_1 := \frac{1}{7} \left( \frac{J_{14} J_1^7}{J_7 J_{2}^7} - 8 \right).
\]
For $\alpha \ge 0$, define recursively
\[
L_{2\alpha+1} := U_A(L_{2\alpha}), \quad L_{2\alpha+2} := U_B(L_{2\alpha+1}).
\]

Using \eqref{Updeffls}, \eqref{u71}, and \eqref{def-a}, one can verify that for $\alpha \ge 0$,
\begin{align*}
L_{2\alpha-1} &= \frac{J_7^3}{J_{14}} \sum_{n=0}^{\infty} a(7^{2\alpha-1} n + \lambda_{2\alpha-1}) q^n, \\
L_{2\alpha} &= \frac{J_{1}^3}{J_2} \sum_{n=0}^{\infty} a(7^{2\alpha} n + \lambda_{2\alpha}) q^n,
\end{align*}
where
\[
\lambda_{2\alpha} = \lambda_{2\alpha+1} = \frac{7^{2\alpha} - 1}{24}.
\]

Following \cite{Pa-Ra12}, we call a map $a : \mathbb{Z} \longrightarrow \mathbb{Z}$ a \textit{discrete function} if it has finite support. Define the sets
\begin{align*}
X_A &:= \left\{ \sum_{k=1}^{\infty} r_1(k) 7^{\left\lfloor \frac{7k+2}{4} \right\rfloor} t^k
        + p_0 \sum_{k=1}^{\infty} r_2(k) 7^{\left\lfloor \frac{7k+2}{4} \right\rfloor} t^k
        + p_1 \sum_{k=0}^{\infty} r_3(k) 7^{\left\lfloor \frac{7k+5}{4} \right\rfloor} t^k \right\}, \\
X_B &:= \left\{ \sum_{k=1}^{\infty} r_1(k) 7^{\left\lfloor \frac{7k-3}{4} \right\rfloor} t^k
        + p_0 \sum_{k=1}^{\infty} r_2(k) 7^{\left\lfloor \frac{7k-3}{4} \right\rfloor} t^k
        + 7 r_3(0) p_1  + p_1 \sum_{k=1}^{\infty} r_3(k) 7^{\left\lfloor \frac{7k}{4} \right\rfloor} t^k \right\},
\end{align*}
where each $r_j$ is a discrete function.

We aim to prove that for $\alpha > 0$:
\begin{equation}
\label{eq:7-L2a}
L_{2\alpha+1} \in 7^\alpha X_B,
\end{equation}
where for a set $X$ and a number $k$,
\[
k X := \{ k x : x \in X \}.
\]

\begin{proof}[Proof of Theorem \ref{thm-main}]
From Appendix \ref{funcr-7}, we observe that in each case there exists an integer $l$ and discrete functions $a_{k,u}^{(i)}(n)$ and $b_{k,u}^{(i)}(n)$ for $l \le k \le l + 6$ such that the following identities hold:
\begin{align}
\nonumber U_A(p_0t^k) &= \sum_{n\geq \lceil k/7 \rceil} a_{k,0}^{(0)}(n) 7^{\left\lfloor\frac{7n-k-2}{4}\right\rfloor} t^n + p_0 \sum_{n\geq \lceil (k+6)/7 \rceil} a_{k,0}^{(1)}(n) 7^{\left\lfloor\frac{7n-k-3}{4}\right\rfloor} t^n \\
\nonumber &\quad + p_1 \sum_{n\geq \lceil (k-1)/7 \rceil} a_{k,0}^{(2)}(n) 7^{\left\lfloor\frac{7n-k+1}{4}\right\rfloor} t^n, \\
\label{eq:UA7k-p1} U_A(p_1t^k) &= \sum_{n\geq \lceil (k+5)/7 \rceil} a_{k,1}^{(0)}(n) 7^{\left\lfloor\frac{7n-k-3}{4}\right\rfloor} t^n + p_0 \sum_{n\geq \lceil (k+5)/7 \rceil} a_{k,1}^{(1)}(n) 7^{\left\lfloor\frac{7n-k-3}{4}\right\rfloor} t^n \\
\nonumber &\quad + p_1 \sum_{n\geq \lceil (k-2)/7 \rceil} a_{k,1}^{(2)}(n) 7^{\left\lfloor\frac{7n-k}{4}\right\rfloor} t^n, \\
\nonumber U_A(t^k) &= \sum_{n\geq \lceil (k+4)/7 \rceil} a_{k,2}^{(0)}(n) 7^{\left\lfloor\frac{7n-k-3}{4}\right\rfloor} t^n + p_0 \sum_{n\geq \lceil (k+5)/7 \rceil} a_{k,2}^{(1)}(n) 7^{\left\lfloor\frac{7n-k-2}{4}\right\rfloor} t^n \\
\nonumber &\quad + p_1 \sum_{n\geq \lceil (k-2)/7 \rceil} a_{k,2}^{(2)}(n) 7^{\left\lfloor\frac{7n-k+1}{4}\right\rfloor} t^n, \\
\nonumber U_B(p_0t^k) &= \sum_{n\geq \lceil k/7 \rceil} b_{k,0}^{(0)}(n) 7^{\left\lfloor\frac{7n-k}{4}\right\rfloor} t^n + p_0 \sum_{n\geq \lceil (k+3)/7 \rceil} b_{k,0}^{(1)}(n) 7^{\left\lfloor\frac{7n-k-1}{4}\right\rfloor} t^n \\
\nonumber &\quad + p_1 \sum_{n\geq \lceil k/7 \rceil} b_{k,0}^{(2)}(n) 7^{\left\lfloor\frac{7n-k+3}{4}\right\rfloor} t^n, \\
\nonumber U_B(p_1t^k) &= \sum_{n\geq \lceil (k+1)/7 \rceil} b_{k,1}^{(0)}(n) 7^{\left\lfloor\frac{7n-k-1}{4}\right\rfloor} t^n + p_0 \sum_{n\geq \lceil (k+4)/7 \rceil} b_{k,1}^{(1)}(n) 7^{\left\lfloor\frac{7n-k-1}{4}\right\rfloor} t^n \\
\nonumber &\quad + p_1 \sum_{n\geq \lceil k/7 \rceil} b_{k,1}^{(2)}(n) 7^{\left\lfloor\frac{7n-k+2}{4}\right\rfloor} t^n, \\
\nonumber U_B(t^k) &= \sum_{n\geq \lceil k/7 \rceil} b_{k,2}^{(0)}(n) 7^{\left\lfloor\frac{7n-k-1}{4}\right\rfloor} t^n.
\end{align}

Using Lemma \ref{lem1} and Lemma \ref{lem2}, we conclude that the above six equations hold for all $k \in \mathbb{N}$.

We now prove \eqref{eq:7-L2a} by induction, establishing the following three claims:
\begin{align*}
&\text{(i) } L_1 \in X_B, \\
&\text{(ii) } g \in X_B \text{ implies } U_B(g) \in X_A, \\
&\text{(iii) } g \in X_A \text{ implies } U_A(g) \in 7 X_B.
\end{align*}

From \eqref{L-1}, we have $L_1 \in X_B$ for some discrete functions $r_i$. Now assume $g \in X_B$, so there exist discrete functions $r_i$ such that
\begin{align*}
g = \sum_{k=1}^{\infty} r_1(k) 7^{\left\lfloor\frac{7k-3}{4}\right\rfloor} t^k + p_0 \sum_{k=1}^{\infty} r_2(k) 7^{\left\lfloor\frac{7k-3}{4}\right\rfloor} t^k + 7 r_3(0) p_1 + p_1 \sum_{k=1}^{\infty} r_3(k) 7^{\left\lfloor\frac{7k}{4}\right\rfloor} t^k.
\end{align*}
Then
\begin{align}
\label{eq:7UBg}
U_B(g) &= \sum_{k=1}^{\infty} r_1(k) 7^{\left\lfloor\frac{7k-3}{4}\right\rfloor} U_B(t^k) + \sum_{k=1}^{\infty} r_2(k) 7^{\left\lfloor\frac{7k-3}{4}\right\rfloor} U_B(p_0 t^k)  \\
&~~~~~~~~~~~~~~~~~~~~~+7 r_3(0) U_B(p_1)+\sum_{k=1}^{\infty} r_3(k) 7^{\left\lfloor\frac{7k}{4}\right\rfloor} U_B(p_1 t^k)\nonumber.
\end{align}
Each sum in \eqref{eq:7UBg} can be expressed in the form $g_1$ for some $g_1 \in X_A$, confirming claim (ii).

Next, assume $g \in X_A$, so there exist discrete functions $r_i$ such that
\begin{align*}
g = \sum_{k=1}^{\infty} r_1(k) 7^{\left\lfloor\frac{7k+2}{4}\right\rfloor} t^k + p_0 \sum_{k=1}^{\infty} r_2(k) 7^{\left\lfloor\frac{7k+2}{4}\right\rfloor} t^k + p_1 \sum_{k=0}^{\infty} r_3(k) 7^{\left\lfloor\frac{7k+5}{4}\right\rfloor} t^k.
\end{align*}
Then
\begin{align}\label{eq:7UBg1}
U_A(g) = \sum_{k=1}^{\infty} r_1(k) 7^{\left\lfloor\frac{7k+2}{4}\right\rfloor} U_A(t^k) + \sum_{k=1}^{\infty} r_2(k) 7^{\left\lfloor\frac{7k+2}{4}\right\rfloor} U_A(p_0 t^k) + \sum_{k=0}^{\infty} r_3(k) 7^{\left\lfloor\frac{7k+5}{4}\right\rfloor} U_A(p_1 t^k).
\end{align}

As the proofs are similar, we focus on the third sum. From \eqref{eq:UA7k-p1}, we have
\begin{align*}
&\sum_{k=0}^{\infty} r_3(k) 7^{\left\lfloor\frac{7k+5}{4}\right\rfloor} U_A(p_1 t^k)
= \sum_{k=0}^{\infty} \sum_{n=1}^{\infty} r_3(k) a_{k,0}^{(0)}(n) 7^{\left\lfloor\frac{7k+5}{4}\right\rfloor + \left\lfloor\frac{7n-k-3}{4}\right\rfloor} t^n \\
&\quad + \sum_{k=0}^{\infty} \sum_{n=1}^{\infty} r_3(k) a_{k,0}^{(1)}(n) 7^{\left\lfloor\frac{7k+5}{4}\right\rfloor + \left\lfloor\frac{7n-k-3}{4}\right\rfloor} p_0 t^n
 + \sum_{k=0}^{\infty} \sum_{n=0}^{\infty} r_3(k) a_{k,0}^{(2)}(n) 7^{\left\lfloor\frac{7k+5}{4}\right\rfloor + \left\lfloor\frac{7n-k}{4}\right\rfloor} p_1 t^n.
\end{align*}
Note that $7 \mid U_A(p_1)$.

\begin{itemize}
    \item For $k=0$:
    \[
    \left\lfloor \frac{7 \cdot 0 + 5}{4} \right\rfloor + 1 = \left\lfloor \frac{5}{4} \right\rfloor + 1 = 1 + 1 = 2 \quad \Rightarrow \quad \left\lfloor \frac{7k+5}{4} \right\rfloor + 1 \geq 2.
    \]

    \item For $k \in \{1, 2\}$ and $n \geq 0$:
    \[
    \left\lfloor \frac{7k+5}{4} \right\rfloor + \left\lfloor \frac{7n - k}{4} \right\rfloor \geq 2.
    \]

    \item For $k \geq 0$:
    \begin{align*}
        \left\lfloor \frac{7k+5}{4} \right\rfloor + \left\lfloor \frac{7n - k - 3}{4} \right\rfloor &\geq 1 + \left\lfloor \frac{7n - 3}{4} \right\rfloor \qquad \text{for  $n \geq 1$},\\
        \left\lfloor \frac{7k+5}{4} \right\rfloor + \left\lfloor \frac{7n - k}{4} \right\rfloor &\geq 1 + \left\lfloor \frac{7n}{4} \right\rfloor\qquad \text{for  $n \geq 0$}.
    \end{align*}
\end{itemize}

Hence, the right-hand side of \eqref{eq:7UBg1} can be written as $7g_1$ for some $g_1 \in X_B$, proving claim (iii). The proof that $g \in X_A$ implies $U_A(g) \in 7X_B$ follows similarly.

\end{proof}

%


\appendix
\section{The Fundamental Relations for the Reciprocal of Crank Parity Function for Powers of $7$}
\label{funcr-7}

\begin{align*}
&\text{Group \uppercase\expandafter{\romannumeral1}}\\
&U_A(p_0) = p_0(-7t + 8 \cdot 7^2 t^2) \\
&\quad + p_1(1 + 7^2 t + 8 \cdot 7^3 t^2 + 7^5 t^3) + (8 \cdot 7 t + 8 \cdot 7^3 t^2 + 8 \cdot 7^4 t^3),\\
&U_A(p_0 t^{-1}) = p_0(8 \cdot 7 t + 7^3 t^2) + p_1(-7 - 7^2 t),\\
&U_A(p_0 t^{-2}) = p_0(-34 \cdot 7 t + 7^5 t^3) + p_1(30 + 7^2 t - 7^4 t^2) + (-8 \cdot 7 t - 4 \cdot 7^3 t^2),\\
&U_A(p_0 t^{-3}) = p_0(80 \cdot 7 t - 88 \cdot 7^3 t^2 - 24 \cdot 7^5 t^3 - 7^7 t^4) \\
&\quad + p_1(-10 \cdot 7 + 72 \cdot 7^2 t + 23 \cdot 7^4 t^2 + 7^6 t^3) + (8 \cdot 7^2 t + 32 \cdot 7^3 t^2),\\
&U_A(p_0 t^{-4}) = p_0(64 \cdot 7^2 t + 1240 \cdot 7^3 t^2 + 48 \cdot 7^6 t^3 + 4 \cdot 7^8 t^4 + 7^9 t^5) \\
&\quad + p_1(-58 \cdot 7 - 152 \cdot 7^3 t - 314 \cdot 7^4 t^2 - 27 \cdot 7^6 t^3 - 7^8 t^4) \\
&\quad + (2 - 52 \cdot 7^2 t - 38 \cdot 7^4 t^2 - 20 \cdot 7^5 t^3 - 2 \cdot 7^7 t^4),\\
&U_A(p_0 t^{-5}) = p_0(-1264 \cdot 7^2 t - 1726 \cdot 7^4 t^2 - 512 \cdot 7^6 t^3 - 64 \cdot 7^8 t^4 - 32 \cdot 7^9 t^5 - 7^{11} t^6) \\
&\quad + p_1(1132 \cdot 7 + 214 \cdot 7^4 t + 468 \cdot 7^5 t^2 + 60 \cdot 7^7 t^3 + 31 \cdot 7^8 t^4 + 7^{10} t^5) \\
&\quad + (-32 + 340 \cdot 7^2 t + 360 \cdot 7^4 t^2 + 12 \cdot 7^7 t^3 + 88 \cdot 7^7 t^4 + 4 \cdot 7^9 t^5),\\
&U_A(p_0 t^{-6}) = p_0(8 + 12732 \cdot 7^2 t + 13992 \cdot 7^4 t^2 + 4532 \cdot 7^6 t^3 + 4752 \cdot 7^7 t^4 + 340 \cdot 7^9 t^5 - 7^{13} t^7) \\
&\quad + p_1(-t^{-1} - 11380 \cdot 7 - 12191 \cdot 7^3 t - 4080 \cdot 7^5 t^2 - 633 \cdot 7^7 t^3 - 48 \cdot 7^9 t^4 - 7^{10} t^5 + 7^{12} t^6) \\
&\quad + (48 \cdot 7 - 312 \cdot 7^3 t - 3180 \cdot 7^4 t^2 - 1056 \cdot 7^6 t^3 - 20 \cdot 7^9 t^4 - 8 \cdot 7^9 t^5 + 4 \cdot 7^{11} t^6),\\[2ex]
&\text{Group \uppercase\expandafter{\romannumeral2}}\\
&U_A(p_1) = p_0(6280 \cdot 7^8 t^5 + 13984 \cdot 7^9 t^6 + 344 \cdot 7^{12} t^7 + 216 \cdot 7^{13} t^8 + 8 \cdot 7^{15} t^9 + 9992 \cdot 7^6 t^4 \\
&\quad + 6472 \cdot 7^4 t^3 - 8 \cdot 7 t + 800 \cdot 7^2 t^2) \\
&\quad + p_1(7 + 13712 \cdot 7^5 t^3 + 16913 \cdot 7^7 t^4 + 184 \cdot 7 t + 3746 \cdot 7^3 t^2 + 70540 \cdot 7^8 t^5 \\
&\quad + 23568 \cdot 7^{10} t^6 + 4684 \cdot 7^{12} t^7 + 79 \cdot 7^{15} t^8 + 7^{18} t^{10} + 36 \cdot 7^{16} t^9) \\
&\quad + (88 \cdot 7 t + 26032 \cdot 7^2 t^2 + 96464 \cdot 7^4 t^3 + 120144 \cdot 7^6 t^4 + 72504 \cdot 7^8 t^5 \\
&\quad + 172216 \cdot 7^9 t^6 + 34848 \cdot 7^{11} t^7 + 600 \cdot 7^{14} t^8 + 40 \cdot 7^{16} t^9 + 8 \cdot 7^{17} t^{10}),\\
&U_A(p_1 t^{-1}) = p_0(-8 \cdot 7 t + 8 \cdot 7^3 t^2) + p_1(8 + 7^3 t + 8 \cdot 7^4 t^2 + 7^6 t^3) + (8 \cdot 7^2 t + 8 \cdot 7^4 t^2 + 8 \cdot 7^5 t^3),\\
&U_A(p_1 t^{-2}) = p_0(64 \cdot 7 t + 8 \cdot 7^3 t^2) + p_1(-57 - 8 \cdot 7^2 t),\\
&U_A(p_1 t^{-3}) = p_0(-320 \cdot 7 t - 8 \cdot 7^3 t^2 + 8 \cdot 7^5 t^3) + p_1(288 + 17 \cdot 7^2 t - 8 \cdot 7^4 t^2) \\
&\quad + (-48 \cdot 7 t - 32 \cdot 7^3 t^2),\\
&U_A(p_1 t^{-4}) = p_0(152 \cdot 7^2 t - 96 \cdot 7^4 t^2 - 200 \cdot 7^5 t^3 - 8 \cdot 7^7 t^4) \\
&\quad + p_1(-138 \cdot 7 + 528 \cdot 7^2 t + 193 \cdot 7^4 t^2 + 8 \cdot 7^6 t^3) + (48 \cdot 7^2 t + 288 \cdot 7^3 t^2),\\
&U_A(p_1 t^{-5}) = p_0(-8 + 176 \cdot 7^2 t + 1496 \cdot 7^4 t^2 + 416 \cdot 7^6 t^3 + 232 \cdot 7^7 t^4 + 8 \cdot 7^9 t^5) \\
&\quad + p_1(t^{-1} - 144 \cdot 7 - 1265 \cdot 7^3 t - 8 \cdot 7^7 t^2 - 32 \cdot 7^7 t^3 - 8 \cdot 7^8 t^4) \\
&\quad + (16 - 344 \cdot 7^2 t - 376 \cdot 7^4 t^2 - 24 \cdot 7^6 t^3 - 16 \cdot 7^7 t^4),\\
&U_A(p_1 t^{-6}) = p_0(160 - 9216 \cdot 7^2 t - 15632 \cdot 7^4 t^2 - 4640 \cdot 7^6 t^3 - 3856 \cdot 7^7 t^4 - 264 \cdot 7^9 t^5 - 8 \cdot 7^{11} t^6) \\
&\quad + p_1(-20 t^{-1} + 1166 \cdot 7^2 + 13460 \cdot 7^3 t + 4274 \cdot 7^5 t^2 + 516 \cdot 7^7 t^3 + 255 \cdot 7^8 t^4 + 8 \cdot 7^{10} t^5) \\
&\quad + (-264 + 2640 \cdot 7^2 t + 3672 \cdot 7^4 t^2 + 704 \cdot 7^6 t^3 + 712 \cdot 7^7 t^4 + 32 \cdot 7^9 t^5),\\[2ex]
&\text{Group \uppercase\expandafter{\romannumeral3}}\\
&U_A(1) = p_0(8 \cdot 7 t + 12 \cdot 7^3 t^2 + 24 \cdot 7^4 t^3 + 2 \cdot 7^6 t^4) \\
&\quad + p_1(-7 - 11 \cdot 7^2 t - 23 \cdot 7^3 t^2 - 2 \cdot 7^5 t^3) + (2 \cdot 7 t + 29 \cdot 7^2 t^2 + 10 \cdot 7^4 t^3 + 7^6 t^4),\\
&U_A(t^{-1}) = p_0(2 \cdot 7 t) + p_1(-2) + 7 t,\\
&U_A(t^{-2}) = p_0(-4 \cdot 7^2 t - 4 \cdot 7^3 t^2) + p_1(26 + 4 \cdot 7^2 t) + (-4 \cdot 7 t - 7^3 t^2),\\
&U_A(t^{-3}) = p_0(26 \cdot 7^2 t + 8 \cdot 7^3 t^2 - 4 \cdot 7^5 t^3) + p_1(-24 \cdot 7 - 12 \cdot 7^2 t + 4 \cdot 7^4 t^2) \\
&\quad + (10 \cdot 7 t - 83 \cdot 7^3 t^2 + 10 \cdot 7^5 t^3),\\
&U_A(t^{-4}) = p_0(-124 \cdot 7^2 t + 32 \cdot 7^4 t^2 + 80 \cdot 7^5 t^3 + 2 \cdot 7^7 t^4) \\
&\quad + p_1(115 \cdot 7 - 23 \cdot 7^3 t - 79 \cdot 7^4 t^2 - 2 \cdot 7^6 t^3) + (1 - 2 \cdot 7^2 t - 83 \cdot 7^3 t^2 + 10 \cdot 7^5 t^3),\\
&U_A(t^{-5}) = p_0(8 + 300 \cdot 7^2 t - 552 \cdot 7^4 t^2 - 20 \cdot 7^7 t^3 - 36 \cdot 7^7 t^4) \\
&\quad + p_1(-t^{-1} - 290 \cdot 7 + 453 \cdot 7^3 t + 136 \cdot 7^5 t^2 + 5 \cdot 7^7 t^3) \\
&\quad + (-18 + 18 \cdot 7^2 t + 150 \cdot 7^4 t^2 + 6 \cdot 7^6 t^3 + 18 \cdot 7^7 t^4 + 7^9 t^5),\\
&U_A(t^{-6}) = p_0(-24 \cdot 7 + 2300 \cdot 7^2 t + 6368 \cdot 7^4 t^2 + 1648 \cdot 7^6 t^3 + 146 \cdot 7^8 t^4 + 80 \cdot 7^9 t^5 + 4 \cdot 7^{11} t^6) \\
&\quad + p_1(3 \cdot 7 t^{-1} - 1949 \cdot 7 - 5406 \cdot 7^3 t - 1551 \cdot 7^5 t^2 - 135 \cdot 7^7 t^3 - 76 \cdot 7^8 t^4 - 4 \cdot 7^{10} t^5) \\
&\quad + (211 - 442 \cdot 7^2 t - 1713 \cdot 7^4 t^2 - 270 \cdot 7^6 t^3 - 323 \cdot 7^7 t^4 - 8 \cdot 7^9 t^5 + 7^{11} t^6),\\[2ex]
&\text{Group \uppercase\expandafter{\romannumeral4}}\\
&U_B(p_0) = p_0(79 \cdot 7 t + 216 \cdot 7^3 t^2 + 80 \cdot 7^5 t^3 + 8 \cdot 7^7 t^4) \\
&\quad + p_1(4 + 34 \cdot 7^3 t + 327 \cdot 7^4 t^2 + 18 \cdot 7^7 t^3 + 19 \cdot 7^8 t^4 + 7^{10} t^5) \\
&\quad + (4 + 1672 \cdot 7 t + 2320 \cdot 7^3 t^2 + 920 \cdot 7^5 t^3 + 144 \cdot 7^7 t^4 + 8 \cdot 7^9 t^5),\\
&U_B(p_0 t^{-1}) = p_0(-8 \cdot 7 t - 7^3 t^2) + p_1(7 + 7^2 t) + 8,\\
&U_B(p_0 t^{-2}) = p_0(-8 \cdot 7^3 t^2 - 7^5 t^3) + p_1(7^3 t + 7^4 t^2) - 12,\\
&U_B(p_0 t^{-3}) = p_0(1 + 10 \cdot 7^2 t + 27 \cdot 7^4 t^2 + 62 \cdot 7^5 t^3 + 6 \cdot 7^7 t^4) \\
&\quad + p_1(-4 \cdot 7 - 3 \cdot 7^4 t - 8 \cdot 7^5 t^2 - 6 \cdot 7^6 t^3) + (80 - 64 \cdot 7^2 t - 32 \cdot 7^4 t^2 - 4 \cdot 7^6 t^3),\\
&U_B(p_0 t^{-4}) = p_0(-6 - 60 \cdot 7^2 t - 162 \cdot 7^4 t^2 - 60 \cdot 7^6 t^3 - 50 \cdot 7^7 t^4 - 7^9 t^5) \\
&\quad + p_1(24 \cdot 7 + 18 \cdot 7^4 t + 54 \cdot 7^5 t^2 + 7^8 t^3 + 7^8 t^4) \\
&\quad + (-96 \cdot 7 + 384 \cdot 7^2 t + 192 \cdot 7^4 t^2 + 24 \cdot 7^6 t^3),\\
&U_B(p_0 t^{-5}) = p_0(3 \cdot 7 + 30 \cdot 7^3 t + 81 \cdot 7^5 t^2 + 30 \cdot 7^7 t^3 + 3 \cdot 7^9 t^4 - 8 \cdot 7^9 t^5 - 7^{11} t^6) \\
&\quad + p_1(-12 \cdot 7^2 - 9 \cdot 7^5 t - 27 \cdot 7^6 t^2 - 3 \cdot 7^8 t^3 + 7^9 t^4 + 7^{10} t^5) \\
&\quad + (752 \cdot 7 - 192 \cdot 7^3 t - 96 \cdot 7^5 t^2 - 12 \cdot 7^7 t^3),\\
&U_B(p_0 t^{-6}) = p_0(2 \cdot 7 + 20 \cdot 7^3 t + 54 \cdot 7^5 t^2 + 20 \cdot 7^7 t^3 + 2 \cdot 7^9 t^4 - 8 \cdot 7^{11} t^6 - 7^{13} t^7) \\
&\quad + p_1(-8 \cdot 7^2 - 6 \cdot 7^5 t - 18 \cdot 7^6 t^2 - 2 \cdot 7^8 t^3 + 7^{11} t^5 + 7^{12} t^6) \\
&\quad + (-108 \cdot 7^3 - 128 \cdot 7^3 t - 64 \cdot 7^5 t^2 - 8 \cdot 7^7 t^3),\\[2ex]
&\text{Group \uppercase\expandafter{\romannumeral5}}\\
&U_B(p_1 t^{-1}) = p_0(552 \cdot 7 t + 216 \cdot 7^4 t^2 + 80 \cdot 7^6 t^3 + 8 \cdot 7^8 t^4) \\
&\quad + p_1(29 + 34 \cdot 7^4 t + 327 \cdot 7^5 t^2 + 18 \cdot 7^8 t^3 + 19 \cdot 7^9 t^4 + 7^{11} t^5) \\
&\quad + (32 + 1672 \cdot 7^2 t + 2320 \cdot 7^4 t^2 + 920 \cdot 7^6 t^3 + 144 \cdot 7^8 t^4 + 8 \cdot 7^{10} t^5),\\
&U_B(p_1 t^{-2}) = p_0(-8 \cdot 7^2 t - 8 \cdot 7^3 t^2) + p_1(7^2 + 8 \cdot 7^2 t) + 40,\\
&U_B(p_1 t^{-3}) = p_0(-8 \cdot 7^4 t^2 - 8 \cdot 7^5 t^3) + p_1(7^4 t + 8 \cdot 7^4 t^2) - 8,\\
&U_B(p_1 t^{-4}) = p_0(8 + 80 \cdot 7^2 t + 216 \cdot 7^4 t^2 + 72 \cdot 7^6 t^3 + 48 \cdot 7^7 t^4) \\
&\quad + p_1(-32 \cdot 7 - 24 \cdot 7^4 t - 65 \cdot 7^5 t^2 - 48 \cdot 7^6 t^3) \\
&\quad + (256 - 512 \cdot 7^2 t - 256 \cdot 7^4 t^2 - 32 \cdot 7^6 t^3),\\
&U_B(p_1 t^{-5}) = p_0(-8 \cdot 7 - 80 \cdot 7^3 t - 216 \cdot 7^5 t^2 - 80 \cdot 7^7 t^3 - 64 \cdot 7^8 t^4 - 8 \cdot 7^9 t^5) \\
&\quad + p_1(32 \cdot 7^2 + 24 \cdot 7^5 t + 72 \cdot 7^6 t^2 + 9 \cdot 7^8 t^3 + 8 \cdot 7^8 t^4) \\
&\quad + (-584 \cdot 7 + 512 \cdot 7^3 t + 256 \cdot 7^5 t^2 + 32 \cdot 7^7 t^3),\\
&U_B(p_1 t^{-6}) = p_0(40 \cdot 7 + 400 \cdot 7^3 t + 1080 \cdot 7^5 t^2 + 400 \cdot 7^7 t^3 + 40 \cdot 7^9 t^4 - 8 \cdot 7^{10} t^5 - 8 \cdot 7^{11} t^6) \\
&\quad + p_1(-160 \cdot 7^2 - 120 \cdot 7^5 t - 360 \cdot 7^6 t^2 - 40 \cdot 7^8 t^3 + 7^{10} t^4 + 8 \cdot 7^{10} t^5) \\
&\quad + (856 \cdot 7^2 - 2560 \cdot 7^3 t - 1280 \cdot 7^5 t^2 - 160 \cdot 7^7 t^3),\\
&U_B(p_1 t^{-7}) = p_0(-136 \cdot 7 - 1360 \cdot 7^3 t - 3672 \cdot 7^5 t^2 - 1360 \cdot 7^7 t^3 - 136 \cdot 7^9 t^4 - 8 \cdot 7^{12} t^6 - 8 \cdot 7^{13} t^7) \\
&\quad + p_1(t^{-1} + 11 \cdot 7^4 + 2798 \cdot 7^4 t + 1167 \cdot 7^6 t^2 + 818 \cdot 7^7 t^3 - 19 \cdot 7^9 t^4 + 6 \cdot 7^{11} t^5 + 8 \cdot 7^{12} t^6) \\
&\quad + (-49632 \cdot 7 + 8392 \cdot 7^3 t + 3984 \cdot 7^5 t^2 + 408 \cdot 7^7 t^3 - 144 \cdot 7^8 t^4 - 8 \cdot 7^{10} t^5),\\[2ex]
&\text{Group \uppercase\expandafter{\romannumeral6}}\\
&U_B(1) = 1,\\
&U_B(t^{-1}) = -4 - 7 t,\\
&U_B(t^{-2}) = 20 - 7^3 t^2,\\
&U_B(t^{-3}) = -88 - 7^5 t^3,\\
&U_B(t^{-4}) = 260 - 7^7 t^4,\\
&U_B(t^{-5}) = 68 \cdot 7 - 7^9 t^5,\\
&U_B(t^{-6}) = -2392 \cdot 7 - 7^{11} t^6.
\end{align*}

\subsection*{Acknowledgements}
The  author was  supported by the National Key R\&D Program of China (Grant No. 2024YFA1014500) and the National Natural Science Foundation of China (Grant No. 12201387).




\end{document}